\documentclass[11pt,oneside]{amsart}
\usepackage[mathscr]{eucal}
\usepackage{amssymb, amsmath, array, amscd}
\usepackage[dvips]{graphicx}
\usepackage{color}
\usepackage{epsfig}
\newtheorem{theorem}{Theorem}[section]
\newtheorem{corollary}{Corollary}[section]

\newtheorem{lemma}{Lemma}[section]

\begin{document}

\title[Compactifications of quotients of the cylinder]{NONALGEBRAIC COMPACTIFICATIONS OF QUOTIENTS OF THE CYLINDER}

\author{Marco Brunella}

\address{Marco Brunella, Institut de Math\'ematiques de Bourgogne
-- UMR 5584 -- 9 Avenue Savary, 21078 Dijon, France}

\begin{abstract} We classify compact complex surfaces which contain a Zariski open subset whose universal covering is the cylinder ${\mathbb D}\times{\mathbb C}$.
\end{abstract}

\maketitle

\section{Introduction}

In this paper we shall be mainly concerned with surfaces of class ${\rm VII}_\circ^+$, that is compact connected smooth complex surfaces $S$ which are minimal and whose Betti numbers are
$$b_1(S)=1 \qquad {\rm and}\qquad b_2(S)>0.$$
According to classical results in Kodaira's classification \cite{Nak2}, every such surface is not K\"ahlerian and its algebraic dimension $a(S)$ is equal to $0$, that is there are no nonconstant meromorphic functions on $S$.

The first examples of class ${\rm VII}_\circ^+$ surfaces were discovered by Inoue around 1974 \cite{Ino1} \cite[\S 4.5]{Nak2}. They are obtained by compactifying the total space of a line bundle of negative degree over an elliptic curve; the compactification is done by gluing a cycle of rational curves, each one of selfintersection $-2$, so that the total selfintersection of the cycle is $0$. Later, this construction was generalized by Enoki \cite{Enok} \cite[\S 4.5]{Nak2}, by replacing the line bundle with an affine one, and by giving a characterization of the  surfaces so obtained. These surfaces are today usually called {\bf Enoki surfaces}.

Other examples of class ${\rm VII}_\circ^+$ surfaces were discovered again by Inoue in \cite{Ino2}, see also \cite[\S 4.1,4.3]{Nak2}. They are constructed by compactifying suitable quotients of the cylinder ${\mathbb D}\times{\mathbb C}$. The compactification is realized by gluing one or two cycles of rational curves, but now  the intersection form of each one of these cycles is negative. This compactification process is close to Hirzebruch's compactification of Hilbert modular surfaces \cite{Hirz}, and for this reason those surfaces constructed by Inoue are nowadays usually called {\bf Inoue-Hirzebruch surfaces}.

A totally different construction of class ${\rm VII}_\circ^+$ surfaces was introduced by Kato \cite{Kato} \cite[\S 5]{Nak2} \cite{Dlou}, via the so-called global spherical shells. It turns out, however, that this class of {\bf Kato surfaces} contains all Enoki surfaces and all Inoue-Hirzebruch surfaces as special cases. Kato surfaces which are not Enoki nor Inoue-Hirzebruch are referred to Kato surfaces of {\bf intermediate type}.

Quite recently, Dloussky, Oeljeklaus and Toma \cite{DOT2} proved a conjecture by Kato concerning a characterization of his surfaces in terms of the existence of sufficiently many rational curves. In the course of their proof, they established the following remarkable fact \cite[\S 3]{DOT2} (see also \cite{DlOe}): if $S$ is a Kato surface of intermediate type and $C\subset S$ is the union of all rational curves in $S$, then the universal covering of $S\setminus C$ is the cylinder ${\mathbb D}\times{\mathbb C}$. Thus, intermediate Kato surfaces can be seen as compactifications of quotients of ${\mathbb D}\times{\mathbb C}$, exactly as Inoue-Hirzebruch surfaces. However, the compactification procedure is here more complicated, because $C$ is not a cycle (or a pair of cycles) of rational curves; it is instead an {\it arboreal} cycle of rational curves, i.e. a cycle to which one or more trees are attached.

On the other side, if $S$ is an Enoki surface then $S\setminus \{$rational curves$\}$ is uniformised by ${\mathbb C}\times{\mathbb C}$, and not by ${\mathbb D}\times{\mathbb C}$. In some sense, intermediate Kato surfaces are closer to Inoue-Hirzebruch surfaces than to Enoki ones.

Our purpose is to show that this cylinder-quotient-compactification procedure cannot go beyond the surfaces already discovered. More generally, we shall work with class ${\rm VII}^+$ surfaces, which are by definition those compact complex surfaces, not necessarily minimal, whose minimal model is of class ${\rm VII}_\circ^+$.

\begin{theorem} \label{main}
Let $S$ be a surface of class ${\rm VII}^+$, and suppose that there exists a curve $C\subset S$ such that $S\setminus C$ is uniformised by the cylinder:
$$\widetilde{S\setminus C} = {\mathbb D}\times{\mathbb C} .$$
Then $S$ is either an Inoue-Hirzebruch surface or a Kato surface of intermediate type, possibly blown up at points belonging to the rational curves.
\end{theorem}

Our proof of this result has a topological flavour. Remark firstly that $S$ has a natural holomorphic foliation ${\mathcal F}$, arising from the vertical foliation in ${\mathbb D}\times{\mathbb C}$. We shall show that the leaves of ${\mathcal F}$ are contained in the leaves of a {\it real} codimension one foliation ${\mathcal H}$, nonsingular on $S\setminus C$ and with a controlled behaviour around $C$. This will be used to compute the Euler characteristic of $S$, and to show that it is equal to the number of rational curves in $C$. Then the conclusion will follow from the main result of \cite{DOT2}. Our real foliation
${\mathcal H}$ is (a posteriori) equal to the foliation given by the level
sets of the Green function of \cite{DlOe}.

As a companion to the theorem above, it would be interesting to prove that Enoki
surfaces, and their blow-ups, are the only surfaces of class ${\rm VII}^+$ which
contain a Zariski open subset isomorphic to a quotient of ${\mathbb C}\times{\mathbb C}$.

Together with standard results \cite{Nak2} (in particular, Inoue's theorem \cite[\S 3]{Nak2}), the theorem above admits the following consequence.

\begin{corollary}\label{corol}
Let $S$ be a compact connected complex surface of algebraic dimension $0$, and suppose that $S$ contains a Zariski open subset uniformised by ${\mathbb D}\times{\mathbb C}$.
Then $S$ is either as in Theorem \ref{main}, or it is an Inoue surface.
\end{corollary}

When $a(S)=1$, the analogous classification problem is quite trivial: every such surface admits an elliptic fibration over a curve, and hence contains a Zariski open subset uniformised by the cylinder, which is saturated by regular fibers of the fibration. When $a(S)=2$, by using Enriques-Kodaira classification (or, better, its logarithmic version, as developed by Sakai, Iitaka, Miyanishi, etc.) one can prove that surfaces with that property are precisely rational, ruled or elliptic ones.

\section{Some preparation}

Let $S$ be a surface of class ${\rm VII}^+$, and let $C\subset S$ be a compact complex curve such that $S\setminus C$ is uniformised by the cylinder ${\mathbb D}\times{\mathbb C}$.

Let us firstly observe that  $S$ cannot be an Enoki surface (or a blow-up of an Enoki surface), because the complement of {\it all} compact complex curves in an Enoki surface is either a ${\mathbb C}$-bundle or a ${\mathbb C}^*$-bundle over an elliptic curve, thus uniformised by ${\mathbb C}^2$ \cite[\S 4.5]{Nak2} (and obviously a larger Zariski open subset cannot be uniformised by ${\mathbb D}\times{\mathbb C}$). From results of Kodaira, Enoki and Nakamura \cite{Enok} \cite[\S 6,7]{Nak2} \cite{Nak1} it then follows that every connected component of $C$ is analytically contractible to one point (its intersection form is negative definite), and moreover it is either a tree of rational curves, or a cycle of rational curves, or an arboreal cycle of rational curves.

The open surface $S\setminus C$ can be presented as
$$S\setminus C = ({\mathbb D}\times{\mathbb C})\big/ \Gamma$$
where $\Gamma\subset {\rm Aut}({\mathbb D}\times{\mathbb C}) $ is the group of deck transformations. Any automorphism of ${\mathbb D}\times{\mathbb C}$ has the form
$$(z,w)\mapsto (\varphi (z), a(z)w+b(z))$$
with $\varphi\in {\rm Aut}({\mathbb D})$, $a\in{\mathcal O}^*({\mathbb D})$ and $b\in{\mathcal O}({\mathbb D})$. We therefore get, by projection, a representation of the group $\Gamma$ into ${\rm Aut}({\mathbb D})$, and $\Gamma$ acts on ${\mathbb D}$ by isometries in the Poincar\'e metric.

We claim that, without loss of generality, we may assume that 
$$no\ orbit\ of\ \Gamma\ on\ {\mathbb D}\ is\ discrete.$$ Indeed, if $\gamma\subset {\mathbb D}$ is a discrete orbit then $(\gamma\times{\mathbb C})/\Gamma$ defines a properly embedded complex curve $L$ in $S\setminus C$. Since every connected component of $C$ is analytically contractible, the closure $\overline L$ of $L$ in $S$ is a compact (rational) curve. But the domain $S\setminus (C\cup \overline L)$ is covered by $({\mathbb D}\setminus\gamma )\times{\mathbb C}$, hence its universal covering is still the cylinder ${\mathbb D}\times{\mathbb C}$. In other words, we can eliminate the discrete orbit by enlarging $C$ to $C\cup\overline L$, and after a finite number of steps we arrive to the desired absence of discrete orbits.

Suppose now that $S$ is not minimal, hence it contains a $(-1)$-curve $E$. If $E$ is not contained in $C$, then $E\cap C$ has cardinality at most $2$, otherwise by contracting $E$ we would obtain a configuration of curves forbidden by \cite{Nak1} (triple points are forbidden). Hence $E\setminus (E\cap C)$ is covered by ${\mathbb C}$, and therefore its preimage in the universal covering of $S\setminus C$ is necessarily of the form $\gamma\times{\mathbb C}$, with $\gamma$ a discrete orbit of $\Gamma$. By the previous paragraph and assumption, this is impossible, and so $E$ must be contained in $C$. But, in that case, we can contract $E$ and we get again a surface of class ${\rm VII}^+$ with a Zariski open subset uniformised by the cylinder. This means that, without loss of generality, we may suppose from now on that 
$$S\ is\ minimal$$ 
that is $S$ is a class ${\rm VII}_\circ^+$ surface.

\begin{lemma}\label{injective}
The above representation $$\Gamma\longrightarrow {\rm Aut}({\mathbb D})$$
is injective.
\end{lemma}

\begin{proof}
Let $H\subset\Gamma$ be the normal subgroup of elements which act identically on ${\mathbb D}$. Since the action of $\Gamma$ on ${\mathbb D}\times{\mathbb C}$ is free and properly discontinuous, and there are no elliptic curves in the quotient, we see that $H$ is either trivial (in which case the proof is done) or infinite cyclic and generated by a translation $(z,w)\mapsto (z,w+1)$ (up to conjugation). In the latter case, the universal covering of $S\setminus C$ factorizes as
$${\mathbb D}\times{\mathbb C}\to {\mathbb D}\times{\mathbb C}^*\to S\setminus C,$$
where the first factor corresponds to the quotient by $H$, so that 
$$S\setminus C = ({\mathbb D}\times{\mathbb C}^*)\big/ \Gamma_0$$
with $\Gamma_0\subset {\rm Aut}({\mathbb D}\times{\mathbb C}^*)$, $\Gamma_0\simeq \Gamma/H$.

Set $\Gamma_0^+=\{ g\in\Gamma_0\ | \ g^*(w\frac{\partial}{\partial w}) = w\frac{\partial}{\partial w}\}$. We have either $\Gamma_0^+=\Gamma_0$, or $\Gamma_0^+$ is a normal subgroup of $\Gamma_0$ of index 2, since every $g\in\Gamma_0$ satisfies  $g^*(w\frac{\partial}{\partial w}) =\pm w\frac{\partial}{\partial w}$. Let us firstly consider the case $\Gamma_0^+=\Gamma_0$. Then we get on $S\setminus C$ a holomorphic vector field $v$, the quotient of $w\frac{\partial}{\partial w}$, generator of a ${\mathbb C}^*$-action. 

We claim that this ${\mathbb C}^*$-action extends holomorphically to $S$. Indeed, consider firstly the normal singular surface $S_0$ obtained by contracting to one point each connected component of $C$, i.e. by compactifying each end of $S\setminus C$ by a single point. The ${\mathbb C}^*$-action clearly extends to $S_0$ (Hartogs), and each singular point is a fixed point of the extended action. Observe now that, since $S$ is minimal, the map $S\to S_0$ coincides with the minimal desingularisation of $S_0$: $C$ contains no $(-1)$-curve. This implies that any ${\mathbb C}^*$-action on $S_0$ can be holomorphically lifted to $S$ (it can be lifted meromorphically, but the absence of $(-1)$-curves in $C$ and the factorization theorem for bimeromorphic maps imply that the lifting is actually holomorphic).

Now, by \cite{Haus} the only surfaces of class ${\rm VII}_\circ^+$ admitting a nontrivial ${\mathbb C}^*$-action are Enoki surfaces (of special type), which is not our case.

When $\Gamma_0^+$ has index 2 in $\Gamma_0$, we then find a double covering of $S\setminus C$ with a ${\mathbb C}^*$-action. By contractibility of $C$, this double covering extends to a (ramified) double covering $S'$ of $S$; denote by $C'$ the preimage of $C$ in $S'$. As before, we obtain that $S'$ is an Enoki surface (if $C'$ contains some $(-1)$-curve, we contract it), hence $S'\setminus C'$ is uniformised by ${\mathbb C}^2$, as well as $S\setminus C$, a contradiction.
\end{proof}  

On $S\setminus C$ we have a holomorphic foliation ${\mathcal F}_0$, arising from the quotient of the vertical foliation of ${\mathbb D}\times{\mathbb C}$, which is $\Gamma$-invariant. By contractibility of $C$, this foliation extends to a (singular) holomorphic foliation ${\mathcal F}$ on $S$. Obviously, every leaf of ${\mathcal F}_0$ is isomorphic either to ${\mathbb C}$ or to ${\mathbb C}^*$, and leaves isomorphic to ${\mathbb C}^*$ correspond to orbits of $\Gamma$ on ${\mathbb D}$ with a nontrivial (cyclic) stabilizer, giving the holonomy of the leaf.

\begin{lemma}\label{leaves}
The foliation ${\mathcal F}_0$ has at most countably many leaves isomorphic to ${\mathbb C}^*$, and each one of these leaves has holonomy generated by an irrational rotation.
\end{lemma}
 
\begin{proof}
Let $L$ be a leaf isomorphic to ${\mathbb C}^*$, say the one corresponding to the orbit of $0\in{\mathbb D}$. We thus have an element $g\in\Gamma$ which acts as $(z,w)\mapsto (\varphi (z),a(z)w+b(z))$, with $\varphi (0)=0$ (i.e., a rotation) and $a(0)=1$, $b(0)=1$ (up to conjugation). Remark that, for every $k\in{\mathbb Z}^*$, $g^k$ is not the identity in $\Gamma$, since $g^k(0,w)=(0,w+k)$. Thus, by the previous lemma, $\varphi^k$ is not the identity, for every $k\in{\mathbb Z}^*$, and this precisely means that $\varphi$ is an irrational rotation. This property implies also the countability of ${\mathbb C}^*$-leaves.
\end{proof}

We refer to \cite{Brun} for standard facts concerning the tangent bundle and the canonical bundle of a singular foliation.

\begin{lemma} \label{nosection}
The canonical bundle $K_{\mathcal F}$ of the foliation ${\mathcal F}$ has no nontrivial global holomorphic section:
$$h^0(S,K_{\mathcal F})=0.$$
\end{lemma}

\begin{proof}
By contradiction, let $s\in H^0(S,K_{\mathcal F})$, $s\not\equiv 0$. Each connected component of its zero set $D=\{ s=0\}$ is contractible to one point, again by Enoki's and Nakamura's results, and if $S_0$ denotes the normal surface obtained after those contractions, then $S\to S_0$ coincides with the minimal desingularization of $S_0$, since $S$ is minimal. Outside of $D$, the canonical bundle $K_{\mathcal F}$ is trivial, hence its dual $T_{\mathcal F}$ is trivial too, and therefore we get a nonvanishing holomorphic vector field $v$ on $S\setminus D$.

This vector field extends to a holomorphic vector field $v_0$ on $S_0$, vanishing at the singular points, by normality (since $S_0$ is singular, it is useful to think to vector fields as derivations of holomorphic functions). This $v_0$ generates a ${\mathbb C}$-action on $S_0$, fixing the singular points, which, as in Lemma \ref{injective}, can be lifted to $S$. Equivalently, the vector field $v$ on $S\setminus D$ holomorphically extends to $S$.
But this means that $h^0(S,T_{\mathcal F})>0$, and together with $h^0(S,K_{\mathcal F})>0$ this implies that $T_{\mathcal F}$ and $K_{\mathcal F}$ are holomorphically trivial.

Therefore $D=\emptyset$ and the above 
vector field on $S$ vanishes only at the (isolated) singularities of the foliation. Now, this contradicts results of \cite{Haus} and \cite[Th.3.2]{DOT1}, according to which any holomorphic vector field on a ${\rm VII}_0^+$-surface must vanish on some curve.
\end{proof}

\section{Construction of a real foliation}

We take the infinite cyclic covering 
$$\pi : \widehat S \longrightarrow S$$
which is provided, in a unique way, by the property $b_1(S)=1$.  We shall denote by $$\Phi : \widehat S \longrightarrow \widehat S$$ a generator of the deck transformations. If we set $$\widehat C = \pi^{-1}(C)$$
then we have
$$\widehat S\setminus\widehat C = ({\mathbb D}\times{\mathbb C})\big/ \widehat\Gamma$$
where $\widehat\Gamma$ is a normal subgroup of $\Gamma$ and $\Gamma / \widehat\Gamma \simeq{\mathbb Z}$. Indeed, the covering $\pi$ corresponds to a homomorphism $\pi_1(S)\to{\mathbb Z}$, which restricts to a homomorphism $\pi_1(S\setminus C)\to {\mathbb Z}$, and then $\widehat\Gamma\subset\Gamma = \pi_1(S\setminus C)$ is precisely the Kernel of the latter. Moreover, $\Gamma / \widehat\Gamma$ is naturally identified with the group of automorphisms of $\widehat S\setminus\widehat C$ generated by the restriction 
$\Phi : \widehat S\setminus\widehat C \to \widehat S\setminus\widehat C$. We will think to $\Phi$ as an element of ${\rm Aut}({\mathbb D}\times{\mathbb C})$, whose action preserves every $\widehat\Gamma$-orbit and sends $\Gamma$-orbits to $\Gamma$-orbits.

Let us denote by $\Gamma_{\mathbb D}$ and $\widehat\Gamma_{\mathbb D}$ the images of $\Gamma$ and $\widehat\Gamma$ into ${\rm Aut}({\mathbb D})$, and note that by Lemma \ref{injective} these are isomorphic images. We have an exact sequence
$$0\to\widehat\Gamma_{\mathbb D} \to \Gamma_{\mathbb D} \to {\mathbb Z} \to 0$$
where ${\mathbb Z}$ is generated by $\varphi\in{\rm Aut}({\mathbb D})$, covered by $\Phi\in{\rm Aut}({\mathbb D}\times{\mathbb C})$.

Let $\omega_{\rm hyp}\in A^{1,1}({\mathbb D})$ be the area form of the Poincar\'e metric on ${\mathbb D}$. Its pull-back to ${\mathbb D}\times{\mathbb C}$ is an ${\rm Aut}({\mathbb D}\times{\mathbb C})$-invariant $(1,1)$-form, hence it induces a smooth $(1,1)$-form on the quotient $S\setminus C$, which
is obviously closed and semipositive (and its Kernel is the foliation ${\mathcal F}_0$). Since $C$ is analytically contractible, this $(1,1)$-form extends, by zero, to a closed positive current $T\in A^{1,1}(S)'$ \cite{Shif}. By construction, $T$ is {\it invariant} by ${\mathcal F}$, that is it can be thought as an integration current along the leaves of ${\mathcal F}$ with respect to an holonomy invariant transverse measure \cite[p.42]{Brun}. Generally speaking, $T$ may fail to be smooth along the curve $C$, even if its generic Lelong number along any irreducible component of $C$ is equal to zero (since the extension is done by zero).

We can pull-back $T$ to $\widehat S$, by pulling back its local potentials, and we get a closed positive current $\widehat T\in A^{1,1}(\widehat S)'$. It is of course equal to the extension by zero of the smooth $(1,1)$-form on $\widehat S\setminus\widehat C$ which arises from the Poincar\'e area form on ${\mathbb D}$. It is invariant by the foliation $\widehat{\mathcal F} = \pi^*({\mathcal F})$.

By construction, the algebraic part of the Siu decomposition of $T$ is identically zero. This implies that $T$ is not only closed but even exact \cite[Rem.8]{Toma}, and moreover $\widehat T$ is $dd^c$-exact \cite[Prop.4]{Toma}\footnote{The statement in \cite{Toma} involves the universal covering of $S$, but it holds also on our covering $\widehat S$ because every closed 1-form on $S$ becomes exact when lifted to $\widehat S$; equivalently, the representation of $\pi_1(S)$ appearing in \cite{Toma} obviously factorizes through $H_1(S,{\mathbb R})$.}. More precisely, there exists a plurisubharmonic function $F$ on 
$\widehat S$ such that
$$\widehat T = dd^cF$$
and 
$$F\circ\Phi = F+\lambda$$
for a suitable nonzero $\lambda\in{\mathbb R}$. Remark that, by elliptic regularity, $F$ is smooth outside $\widehat C$. A posteriori, $F$ will be something like $-\log (-H)$, where $H$ is the Green function of \cite{DlOe}.

\begin{lemma} \label{constant}
The function $F$ is constant along the leaves of $\widehat{\mathcal F}$.
\end{lemma}

\begin{proof}
Let us consider $\partial_{\widehat{\mathcal F}}F = \partial F\vert_{\widehat{\mathcal F}}$ around a nonsingular point of the foliation. In local coordinates $(z,w)$ such that the foliation is given by $dz=0$, that $(1,0)$-form on the leaves is expressed by $F_wdw$. But, in the same coordinates, $\widehat T = \mu\cdot{\rm i}dz\wedge d\overline z$, since it is ${\widehat{\mathcal F}}$-invariant, and therefore for its potential $F$ the partial derivatives  $F_{w\overline w}$, $F_{w\overline z}$ and $F_{z\overline w}$ are all identically zero. That is, the function $F_w$ is holomorphic.
Thus $\partial_{\widehat{\mathcal F}}F$ defines a holomorphic section of $K_{\widehat{\mathcal F}}$ outside the singularities of the foliation, and consequently everywhere by Hartogs.

Since $F\circ\Phi = F + \lambda$, that section descends to $S$, as a holomorphic section of $K_{\mathcal F}$. By Lemma \ref{nosection}, this last section must be identically zero, and by construction this means that $F$ is constant along the leaves.
\end{proof}

As a consequence of this lemma, the function $F\vert_{\widehat S\setminus\widehat C}$ lifts to the universal covering ${\mathbb D}\times{\mathbb C}$ as a $\widehat\Gamma$-invariant smooth function depending only on the first variable. Thus, we get a $\widehat\Gamma_{\mathbb D}$-invariant function $f\in C^\infty ({\mathbb D})$ such that 
$$dd^cf = \omega_{\rm hyp}$$
and moreover, from $F\circ\Phi = F+\lambda$,
$$f\circ\varphi = f+\lambda .$$

Let now $\widehat{\mathcal H}$ be the real codimension one foliation on $\widehat S\setminus\widehat C$ defined by the level sets of $F$. It is clearly $\Phi$-invariant, and hence it descends to a real codimension one foliation ${\mathcal H}$ on $S\setminus C$.
The leaves of ${\mathcal F}_0$ are contained in those of ${\mathcal H}$.

\begin{lemma}\label{smooth}
${\mathcal H}$ is a nonsingular foliation on $S\setminus C$.
\end{lemma}

\begin{proof}
This is equivalent to say that the foliation on ${\mathbb D}$ defined by $f$ is nonsingular.
From the $\widehat\Gamma_{\mathbb D}$-invariance of $f$ and $f\circ\varphi = f+\lambda$, it follows that the critical set ${\rm Crit}(f)$ of $f$ is $\Gamma_{\mathbb D}$-invariant. Remark that $f$ is not only smooth, but even real analytic, hence ${\rm Crit}(f)$ is a real analytic subset, every connected component of which is either a single point or a real analytic curve. Since $\Gamma_{\mathbb D}$ has no discrete orbits, ${\rm Crit}(f)$ has no isolated points, and for the same reason each connected component is free of singularities and it is a smooth real analytic curve. It follows from this that the level sets of $f$ give a smooth (and even real analytic) foliation of the disc.
\end{proof}

In the next section we shall analyse the behaviour of ${\mathcal H}$ around $C$, and this will lead to the proof of Theorem \ref{main}.

\section{Qualitative structure of the real foliation}

We can now better describe the structure of $C$.

\begin{lemma}\label{notree}
No connected component of $C$ is a tree of rational curves.
\end{lemma}

\begin{proof}
By contradiction, let $D$ be a connected component of $C$ which is a tree of rational curves. According to \cite{Cama} (see also \cite[p.40]{Brun}), the foliation ${\mathcal F}$ has a {\it separatrix} based at some point of $D$: there exists $p\in D$ and an analytic curve $A$ in some neighbourhood of $p$, passing through $p$, such that:
\begin{enumerate}
\item[(i)] $A\cap D=\{ p\}$;
\item[(ii)] $A\setminus\{ p\}$ is contained in a leaf of ${\mathcal F}$ (i.e., of ${\mathcal F}_0$).
\end{enumerate}
This leaf $L$ of ${\mathcal F}_0$ cannot be isomorphic to ${\mathbb C}$, otherwise $L\cup\{ p\}$ would be a rational curve, hence $L$ would be properly embedded in $S\setminus C$ and $\Gamma_{\mathbb D}$ would have a discrete orbit, contradicting our general assumptions.
Therefore $L$ is isomorphic to ${\mathbb C}^*$, and by Lemma \ref{leaves} its holonomy has infinite order, generated by an irrational rotation.

However, if $T\simeq{\mathbb D}$ is a disc in $S\setminus C$ transverse to $L$, on which the holonomy acts, then the trace of the real foliation ${\mathcal H}$ on $T$ is a {\it nonsingular} real foliation  which is invariant by the holonomy. It is very easy to see that an irrational rotation cannot preserve a nonsingular real foliation, and this contradiction proves the lemma.
\end{proof}

Remark that the argument of the previous proof shows, more generally, that no leaf of ${\mathcal F}_0$ is isomorphic to ${\mathbb C}^*$. Therefore, every element of $\Gamma$, different from the identity, acts on ${\mathbb D}$ by a parabolic or hyperbolic automorphism. It may be useful at this point to look at \cite{Ino2} and \cite[\S 4]{DOT2}, where the action of $\Gamma$ is explicitely described in the case of Inoue-Hirzebruch surfaces and Kato surfaces of intermediate type.

By Lemma \ref{notree} and Nakamura's results \cite{Nak1}, every connected component of $C$ is a (arboreal) cycle of rational curves. Moreover, and still by \cite{Nak1}, the (cyclic) fundamental group of $C$ injects into $\pi_1(S)$. This means that every connected component of $\widehat C$ is a (arboreal) infinite chain of rational curves, accumulating to both ends of $\widehat S$. Actually, we may suppose from now on that $C$ and $\widehat C$ are both connected, for in the opposite case the paper \cite{Nak1} already shows that $S$ is an Inoue-Hirzebruch surface.

The plurisubharmonic function $F$ on $\widehat S$ is necessarily constant on $\widehat C$, by the maximum principle. But we also have $F\circ\Phi = F+\lambda$, with $\lambda\not= 0$, and $\Phi$ acts on $\widehat C$ as a translation. We therefore conclude that $F$ is equal to $-\infty$ on $\widehat C$, and more precisely
$$\widehat C = \{ F=-\infty\}$$
since $F$ is smooth outside $\widehat C$.

In order to understand better the structure of ${\mathcal H}$ around $C$, it is convenient to resolve the singularities of ${\mathcal F}$, in Seidenberg's sense \cite{Brun}. Denote by 
$$r: S'\to S$$ such a resolution, and set $C'=r^{-1}(C)$, ${\mathcal F}'=r^*({\mathcal F})$. Note that $C'$ is still a (arboreal) cycle of rational curves. 

Take a point $p\in C'$ which is singular for ${\mathcal F}'$. Then, around $p$, the foliation ${\mathcal F}'$ has a plurisubharmonic first integral $G$ (derived from $F$), which is polar along $C'$ and defines a smooth nonsingular foliation outside $C'$. Each branch of $C'$ through $p$ is necessarily a separatrix of ${\mathcal F}'$, and conversely every separatrix is contained in $C'$ (since $G$ is constant on the separatrix and equal to $-\infty$ at $p$). The holonomy of such a separatrix, evaluated on some transverse disc $T$, admits a subharmonic first integral $G\vert_T$, which has a pole at $t=T\cap C'$ and defines a smooth nonsingular foliation outside $t$. The level sets of such a function are necessarily circles around $p$, in some neighbourhood of $p$ (here the smoothness of the foliation is essential). It follows easily that the holonomy of the separatrix is conjugate to a rotation, rational or irrational. 

Remark also that $p$ cannot be a singularity of saddle-node type, whose strong separatrix never has holonomy conjugate to a rotation. By standard linearization results \cite{Brun}, the foliation ${\mathcal F}'$ around $p$ is then expressed by the equation
$$zdw+awdz=0$$
with $a$ real and positive, in suitable coordinates $(z,w)$. 

When $a$ is irrational, the leaves of ${\mathcal F}'$ in this local chart are dense in the real hypersurfaces given by the level sets of $|w||z|^a$. It follows that $G$ factorizes through that function, i.e. $G(z,w)=h(|w||z|^a)$ for some smooth $h:(0,\alpha )\to{\mathbb R}$, increasing  and satisfying $h(t)\to -\infty$ as $t\to 0$ (of course, $h$ satisfies some convex-type relation, since $G$ is plurisubharmonic). When $a$ is rational, say $a=m/n$, then ${\mathcal F}'$ admits the local holomorphic first integral $z^mw^n$. In that case $G$ factorizes as $G(z,w)=g(z^mw^n)$ for some subharmonic $g:{\mathbb D}(\alpha )\to {\mathbb R}$, having a pole at $0$ and whose level sets are circles around $0$.

In both cases, we see that the sublevel sets $\{ G < s\}$, with $s$ close to $-\infty$, provide a fundamental system of tubular neighbourhoods of $\{ zw=0\}$. Remark also that the point $p$ is necessarily a nodal point of $C'$, i.e. we have ${\rm Sing}({\mathcal F}')= {\rm Sing}(C')$.

Of course, a similar property holds when $p\in C'$ is regular for ${\mathcal F}'$: the sublevel sets of the local plurisubharmonic first integral $G$  (derived from $F$) give a fundamental system of tubular neighbourhoods of the leaf through $p$, which is contained in $C'$.

\begin{lemma}\label{euler}
There exists a tubular neighbourhood $V$ of $C'$ in $S'$ such that the Euler characteristic of $S'\setminus V$ vanishes:
$$\chi (S'\setminus V)=0.$$
\end{lemma}

\begin{proof}
Fix a Riemannian metric on $S'$. Set ${\mathcal H}'=r^*({\mathcal H})$, which is a
nonsingular foliation on  $S'\setminus C'$ whose behaviour around every point of $C'$ is described by the previous considerations. Let ${\mathcal L}$ be the real one-dimensional nonsingular foliation on $S'\setminus C'$ orthogonal to ${\mathcal H}'$. Then we can easily find a tubular neighbourhood $V$ of $C'$ whose boundary $\partial V$ is transverse to ${\mathcal L}$. That is, ${\mathcal L}\vert_{S'\setminus V}$ is nonsingular and transverse to the boundary, and we get the conclusion by Poincar\'e-Hopf formula.
\end{proof}

Let $\ell$ be the number of rational curves in $C'$. Since $C'$ is a (arboreal) cycle of rational curves, we have
$$\chi (V)=\ell$$
and hence, by additivity of Euler characteristics,
$$\chi (S')=\ell $$
which can be rewritten as 
$$b_2(S')=\ell$$
because, by $b_1(S')=1$, the Euler characteristic of $S'$ coincides with its second Betti number.

We deduce that the surface $S'$ contains (at least) $b_2(S')$ rational curves, and $S$ contains (at least) $b_2(S)$ rational curves. By \cite{DOT2}, $S$ is a Kato surface, and the proof of Theorem \ref{main} is complete (the last part, concerning the location of blow-ups, is discussed in the next section).

\section{Algebraic dimension zero}

We give here the deduction of Corollary \ref{corol}, as well as some complements to Theorem \ref{main}.

Let us observe a general fact. Let $S$ be a compact connected surface and let $U\subset S$
be a Zariski open subset uniformised by ${\mathbb D}\times{\mathbb C}$. Let $\Omega\subset S$ be a strictly pseudoconvex domain whose boundary $M=\partial\Omega$ has finite fundamental group and is fully contained in $U$:
$$M\subset U.$$
We claim that $\Omega$ too is fully contained in $U$:
$$\Omega\subset U.$$
Indeed, $M$ is finitely covered by some compact connected hypersurface $\widetilde M\subset {\mathbb D}\times{\mathbb C}$, which separates the ambient space into two connected components: $({\mathbb D}\times{\mathbb C})\setminus\widetilde M = V_1\cup V_2$, with $V_1$ bounded and $V_2$ unbounded. This hypersurface has a pseudoconvex side and a pseudoconcave one, and since a bounded domain in ${\mathbb C}^2$ cannot be pseudoconcave, we deduce that $V_1$ is pseudoconvex. The projection of $V_1$ to $U$ is a relatively compact pseudoconvex domain $\Omega_1\subset U$ bounded by $M$, and obviously $\Omega$ must coincide with it.

This shows, for instance, that the Zariski closed subset $C=S\setminus U$ is purely one-dimensional, i.e. there are no isolated points in $C$. More generally, $U$ cannot have a pseudoconcave end with finite fundamental group.

Suppose now that the algebraic dimension $a(S)$ is $0$, and let $D\subset S$ be the maximal divisor, that is the union of all compact complex curves. Let $D_j$ be a connected component of $D$, and suppose that it is analytically contractible to one point, so that we can find a strictly pseudoconvex tubular neighbourhood $\Omega_j$ of $D_j$, with boundary $M_j=\partial\Omega_j$ disjoint from $D$. We claim that (still assuming the existence of  $U\subset S$ universally covered by the cylinder)
$$| \pi_1(M_j)| = \infty .$$
To see this, observe firstly that $M_j$ is necessarily contained in $U$, because $C\subset D$ and $M_j\cap D=\emptyset$. Hence, if $| \pi_1(M_j)| < \infty$ then $\Omega_j$ is contained in $U$, and in particular $D_j$ is contained in $U$. But $| \pi_1(M_j)| < \infty$ implies also that $D_j$ is a tree of rational curves, and of course $U$ cannot contain a rational curve. This contradiction proves the claim.

It is here worth observing that the implication ``$| \pi_1(M_j)| < \infty$ $\Rightarrow$ $D_j$ is a tree of rational curves'' is far from being reversible, even when $D_j$ contracts to a rational singularity (which is the case mostly occurring \cite{Nak1}). This is the reason for which in Lemma \ref{notree} we used a very different argument.

According to known results \cite{Nak2} \cite{Tele}, any compact connected complex surface $S$ with $a(S)=0$ has a (unique) minimal model $S_0$ belonging to the following list:
\begin{enumerate}
\item[(1)] tori of algebraic dimension $0$;
\item[(2)] K3 surfaces of algebraic dimension $0$;
\item[(3)] Hopf surfaces of algebraic dimension $0$;
\item[(4)] Inoue surfaces;
\item[(5)] Surfaces of class ${\rm VII}_\circ^+$.
\end{enumerate}

The last case is the object of Theorem \ref{main}. Let us here observe that if $S$ is not minimal (and contains $U$ uniformised by the cylinder) then the blow-ups are necessarily done over points of $S_0$ which belong to the maximal divisor: otherwise, the maximal divisor of $S$ would contain a contractible component $D_j$ with simply connected link $M_j$, a contradiction with the previous remarks.

A Inoue surface is, by definition, a quotient of ${\mathbb D}\times{\mathbb C}$, hence it
enters into Corollary \ref{corol}. The maximal divisor of a Inoue surface is empty; we deduce from this, by the same argument as before, that a blown up Inoue surface cannot contain a Zariski open subset uniformised by the cylinder.
 
A Hopf surface with algebraic dimension $0$ contains one or two elliptic curves, and no more compact curves. The complement of the curve(s) is uniformised by ${\mathbb C}^2$. It is then easily shown that such a surface, even blown up, cannot contain a Zariski open subset uniformised by the cylinder. Remark that this argument was already used to exclude Enoki surfaces at the beginning of the proof of Theorem \ref{main}, and indeed Enoki surfaces are really very similar to Hopf surfaces.
 
The case of tori is excluded in Corollary \ref{corol} because a torus of algebraic dimension $0$ has empty maximal divisor, and we proceed as in the Inoue case.

Finally, the case of K3 surfaces is excluded by the same type of arguments and the following standard fact: every connected component of the maximal divisor of a K3 surface of algebraic dimension $0$ is a tree of $(-2)$-curves whose contraction produces a A-D-E singularity, whose link has finite fundamental group.


\begin{thebibliography}{30}
\bibitem[Brun]{Brun} M. Brunella, {\sl Birational geometry of foliations}, Publ. Mat. IMPA (2004)
\bibitem[Cama]{Cama} C. Camacho, {\sl Quadratic forms and holomorphic foliations on singular surfaces}, Math. Ann. 282 (1988), 177-184
\bibitem[Dlou]{Dlou} G. Dloussky, {\sl Structure des surfaces de Kato}, M\'em. Soc. Math. France 14 (1984)
\bibitem[DlOe]{DlOe} G. Dloussky, K. Oeljeklaus, {\sl Vector fields and foliations on compact surfaces of class ${\rm VII}_\circ$}, Ann. Inst. Fourier 49 (1999), 1503-1545
\bibitem[DOT1]{DOT1} G. Dloussky, K. Oeljeklaus, M. Toma, {\sl Surfaces de la classe ${\rm VII}_\circ$ admettant un champ de vecteurs}, Comment. Math. Helv. 75 (2000), 255-270
\bibitem[DOT2]{DOT2} G. Dloussky, K. Oeljeklaus, M. Toma, {\sl Class ${\rm VII}_\circ$ surfaces with $b_2$ curves}, Tohoku Math. J. (2) 55 (2003), 283-309
\bibitem[Enok]{Enok} I. Enoki, {\sl Surfaces of class ${\rm VII}_\circ$ with curves}, Tohoku Math. J. (2) 33 (1981), 453-492
\bibitem[Haus]{Haus} J. Hausen, {\sl Zur Klassifikation glatter kompakter ${\mathbb C}^*$-Fl\"achen}, Math. Ann. 301 (1995), 763-769
\bibitem[Hirz]{Hirz} F. Hirzebruch, {\sl Hilbert modular surfaces}, Monographie 21 de L'Enseignement Math\'ematique (1973)
\bibitem[Ino1]{Ino1} M. Inoue, {\sl New surfaces with no meromorphic functions},  Proceedings of the ICM (Vancouver, 1974), Vol. 1, 423-426
\bibitem[Ino2]{Ino2} M. Inoue, {\sl New surfaces with no meromorphic functions II}, in Complex Analysis and Algebraic Geometry, Iwanami Shoten (1977), 91-106
\bibitem[Kato]{Kato} M. Kato, {\sl Compact complex manifolds containing ``global spherical shells''}, Proceedings of the Int. Symp. Algebraic Geometry (Kyoto, 1977), 45-84
\bibitem[Nak1]{Nak1} I. Nakamura, {\sl On surfaces of class ${\rm VII}_\circ$ with curves},  Invent. Math. 78 (1984), 393-443
\bibitem[Nak2]{Nak2} I. Nakamura, {\sl Towards classification of non-K\"ahlerian complex surfaces}, Sugaku Expositions 2 (1989), 209-229
\bibitem[Shif]{Shif} B. Shiffman, {\sl Extension of positive line bundles and meromorphic maps}, Invent. Math. 15 (1972), 332-347
\bibitem[Tele]{Tele} A. Teleman, {\sl Projectively flat surfaces and Bogomolov's theorem on class ${\rm VII}_\circ$ surfaces}, Internat. J. Math. 5 (1994), 253-264
\bibitem[Toma]{Toma} M. Toma, {\sl On the K\"ahler rank of compact complex surfaces},  Bull. Soc. Math. France 136 (2008), 243-260
\end{thebibliography}
\end{document}